\documentclass[oneside,11pt]{amsart}%
\usepackage{amsthm,amsmath,amssymb,amsbsy,bbm,mathrsfs,supertabular}
\usepackage[authoryear]{natbib}

\RequirePackage[colorlinks,citecolor=blue,urlcolor=blue]{hyperref}

\newcommand{\bm}[1]{\mbox{\boldmath $#1$}}
\newcommand{\str}[1]{\rule{0mm}{#1mm}}
\DeclareMathAlphabet{\mathonebb}{U}{bbold}{m}{n}
\newcommand{\1}{\ensuremath{\mathonebb{1}}}
\def\Es{{\mathbb{E}}}
\def\E{{\mathbb{E}_s}}

\def\IL{{\mathbb{L}}}
\def\N{{\mathbb{N}}}
 
\def\R{{\mathbb{R}}}
\def\P{{\mathbb{P}}}

\def\II{{\mathcal I}}
\def\JJ{{\mathcal J}}

\def\cP{\mathcal P}

\def\A{{\mathscr A}}
\def\B{{{\mathscr B}}}
\def\CC{\mathscr C}
\def\D{{\mathscr{D}}}

\def\FF{{\mathscr{F}}}
\def\GG{{\mathscr{G}}}

\def\LL{{\mathscr L}}

\def\X{{\mathscr{X}}}
\newcommand{\bs}[1]{\boldsymbol{#1}}

\def\gh{\mathbf{h}}

\def\gs{{\mathbf{s}}}
\def\gt{{\mathbf{t}}}
\def\gT{{\mathbf{T}}}

\def\w{\mathbf{w}}
\def\gX{{\bs X}}

\def\g0{{\mathbf{0}}}

\def\gup{\boldsymbol{\Upsilon}}

\def\eps{{\varepsilon}}

\def\<{{\langle}}
\def\>{{\rangle}}

\newcommand{\eref}[1]{(\ref{#1})}
\newcommand{\pa}[1]{\left({#1}\right)}
\newcommand{\norm}[1]{\left\|{#1}\right\|}
\newcommand{\cro}[1]{\left[{#1}\right]}
\newcommand{\ab}[1]{\left|{#1}\right|}
\newcommand{\ac}[1]{\left\{{#1}\right\}}

\newtheorem{thm}{Theorem}
\newtheorem{lem}{Lemma}
\newtheorem{prop}{Proposition}
\newtheorem{cor}{Corollary}
\newtheorem{defi}{Definition}

\newtheorem{remark}{Remark}

\def\1{1\hskip-2.6pt{\rm l}}

\begin{document}
\title{rho-estimators for shape restricted density estimation}
\author{Y. Baraud}
\address{Univ. Nice Sophia Antipolis, CNRS,  LJAD, UMR 7351, 06100 Nice, France.}
\email{baraud@unice.fr}
\author{L. Birg\'e}
\address{Sorbonne Universit\'es, UPMC Univ.\ Paris 06, CNRS UMR 7599, LPMA, 75005 Paris, France.}
\email{lucien.birge@upmc.fr}
\date{\today}
\begin{abstract}
The purpose of this paper is to pursue our study of $\rho$-estimators built from i.i.d.\ observations that we defined in
Baraud~{\it et al.}~\citeyearpar{Baraud:2014qv}. For a $\rho$-estimator based on some 
model $\overline{S}$ (which means that the estimator belongs to $\overline{S}$) and a true distribution of the observations that also belongs to $\overline{S}$, the risk (with squared Hellinger loss) is bounded by a quantity
which can be viewed as a dimension function of the model and is often related to the ``metric dimension'' of this model, as defined in Birg\'e~\citeyearpar{MR2219712}. This is a minimax point of view and it is well-known that it is pessimistic. Typically, the bound is accurate for most points in the model but may be very pessimistic when the true distribution belongs to some specific part of it. This is the situation that we want to investigate here. For some models, like the set of decreasing densities on $[0,1]$, there exist specific points in the model that we shall call {\em extremal} and for which the risk is substantially smaller than the typical 
risk. Moreover, the risk at a non-extremal point of the model can be bounded by the sum of the risk bound at a well-chosen extremal point plus the square of its distance to this point. This implies that if the true density is close enough to an extremal point, the risk at this point may be smaller than the minimax risk on the model and this actually remains true even if the true density does not belong to the model. The result is based on some refined bounds on the suprema of empirical processes that are established in Baraud~\citeyearpar{Baraud:2014kq}.
\end{abstract}
\maketitle
\section{Introduction}
The present paper pursues the study of $\rho$-estimation, introduced in Baraud~{\it et al.}~\citeyearpar{Baraud:2014qv}, as a versatile estimation strategy based on models. We want here to explain some specific property of these estimators that we shall call {\em superminimaxity}, a study which was motivated by a conference that Adityanand Guntuboyina gave in Cambridge in June 2014. His talk was about Gaussian regression but we shall deal here with density estimation. Given $n$ i.i.d.\ observations $X_{1},\ldots,X_{n}$ with an unknown density $s$ with respect to some reference measure $\mu$ and an estimator $\widehat{s}(X_{1},\ldots,X_{n})$ of $s$, we measure its performance using the loss function $h^{2}(s,\widehat{s})$ where $h$ is the Hellinger distance. We shall focus here on $\rho$-estimators and some of their properties that lead to superminimaxity. 

The first of these properties is {\em robustness}. There exist various notions of robustness:  robustness to model contamination, robustness to possible outliers, etc.\ --- see Huber~\citeyearpar{MR606374} for some illustrations ---. In some of these cases, the problem can be formulated in the following way. If we know the performance of an estimator when the true density $s=\overline s$ belongs to a model $\overline S$, how does it deteriorate when $s$ is actually of the form $(1-\eps)\overline s+\eps t$ for some small $\eps\in (0,1)$ and an arbitrary density $t\neq \overline s$, that is, when a proportion $\eps$ of the data actually corresponds to a sample of density $t$ and not $\overline s$. Since for such a density $s$ one can check that $h^{2}(s, \overline s)\le \eps$, it is natural to wonder what happens to the risk of the estimator not only when $s$ is a mixture of the form $(1-\eps)\overline s+\eps t$ as before but also, more generally, when it belongs to a small Hellinger ball around $\overline s$, which leads to the notion of {\em robustness with respect to Hellinger deviation}s that we shall use here. 

To illustrate the problem of contamination, assume that we choose as our statistical model $\overline{S}$ for the unknown density $s$ the set of uniform densities on $[0,\theta]$ with $0<\theta\le10$, in which case the MLE (maximum likelihood estimator) is the uniform density $\widehat{s}$ on $\left[0,X_{(n)}\right]$, where $X_{(n)}$ is the largest observation, with a risk $\E\!\left[h^{2}(s,\widehat{s})\right]$ bounded by $C/n$ when $s$ belongs to $\overline{S}$, $\Es_{t}$ denoting the expectation when the true density is $t$. But what if the true density $s$ does not belong to $\overline{S}$? Unfortunately, the situation may become quite different. If $s$ is the mixture $s_{0}=(1-1/n)\1_{[0,1]}+(1/n)(\1_{[9,10]})$ for some $n$ larger than 100, it is easy to check that, with probability of order $1-e^{-1}\approx 0.63$, at least one of the $X_{i}$ will be larger than 9 and the MLE will be the uniform distribution on $\left[0,X_{(n)}\right]$ with $X_{(n)}>9$, which is a terrible estimator of $s=s_{0}$, although the model $\overline{S}$ is quite good since the Hellinger distance between $s$ and $\overline S$ is not larger than $1/\sqrt{n}$. 

The previous example shows that the MLE is definitely not robust in our sense since it may be very sensitive to small deviations from the model on the contrary to $\rho$-estimators. To be more precise, let us consider some model of densities $\overline{S}$ and a $\rho$-estimator $\widehat{s}$ based on $\overline{S}$ with a risk function on $\overline{S}$ bounded by $R(\overline{s},n)$, that is,
\begin{equation}
\Es_{\overline{s}}\left[h^{2}(\overline{s},\widehat{s})\right]\le R(\overline{s},n)\quad\mbox{for all }\overline{s}\in\overline{S}.
\label{eq-localrisk}
\end{equation}
The robustness of $\rho$-estimators can be expressed by the following property, proven in Baraud~{\it et al.}~\citeyearpar{Baraud:2014qv}: whatever the density $s$,
\begin{equation}\label{eq-robust3}
\E\!\left[h^{2}(s,\widehat{s})\right]\le C_{0}\left[R(\overline{s},n)+h^{2}(s,\overline{s})\right]\quad\mbox{for all }\overline{s}\in\overline{S},
\end{equation}
where $C_{0}$ is a universal positive constant. This is a fondamental property of $\rho$-estimators for the following reasons: if $s$ is quite close to a simple density $\overline s$ in $\overline S$ which can be estimated with a small risk bound $R(\overline s,n)$, the $\rho$-estimator will essentially behave as if the true density were $\overline s$ and the risk bound at $s$ will be that at $\overline s$ plus a small additional term that can be viewed as a squared bias. Intuitively, a $\rho$-estimator based on a sample with density $s$ and a $\rho$-estimator based on a sample with density $\overline s$ will remain close. In our parametric example based on uniform distributions, everything happens as if the $\rho$-estimator only considered the data with values in $[0,1]$ and ignored the others. Consequently, its risk remains of order $1/n$ even when $s=s_{0}$ instead of $s=\overline s=1_{[0,1]}$. This notion of robustness is quite flexible and shows that the risk of the estimator does not deteriorate much in a small Hellinger neighbourhood of any point $\overline s$ of the model. 

For many well-chosen models $\overline{S}$, the risk can be uniformly bounded on $\overline{S}$:
\begin{equation}
\sup_{\overline{s}\in\overline{S}}R(\overline{s},n)\le R(\overline{S},n),
\label{eq-robust}
\end{equation}
which corresponds to the minimax point of view, so that (\ref{eq-robust}) leads to
\begin{equation}\label{eq-robustG}
\E\!\left[h^{2}(s,\widehat{s})\right]\le C_{0}\left[R(\overline{S},n)+h^{2}(s,\overline{S})\right]\quad\mbox{whatever the density }s.
\end{equation}
It turns out that for some models $\overline{S}$, there exists a subset $V$ of $\overline{S}$ such that the risk bounds $R(\overline s,n)$ are substantially smaller than $R(\overline{S},n)$ for all 
$\overline s\in V$. This is what we call {\em superminimaxity}. Although there exists some analogy in the denomination, the notion is quite different from the one of superefficiency as described in the famous counterexample of Hodges and the Theorem of Le Cam about points of superefficiency, apart from the fact that it deals with the property that estimation is faster at some points. However superefficiency is an asymptotic property at a point, while superminimaxity on $V$ is definitely nonasymptotic and defined for a given value of the number $n$ of observations. For a detailed study of superefficiency, one could look at the paper by Brown, Low, and Zhao~\citeyearpar{MR1604424}. Moreover, superminimaxity on $V$ together with the robustness of $\rho$-estimators has the following consequence: if $s$ is either in $V$ or close enough to it, the risk bound at $s$,
\[
\E\!\left[h^{2}(s,\widehat{s})\right]\le C_{0}\inf_{\overline{s}\in V}\left[R(\overline{s},n)+h^{2}(s,\overline{s})\right],
\]
may be substantially smaller than the typical risk bound~\eref{eq-robustG} leading to superminimaxity at $s$. It is actually the combination of the robustness of $\rho$-estimators and the existence of local risk bounds of the form (\ref{eq-localrisk}) that lead to this phenomenon as also described, in a different framework, by Chatterjee {\em et al.}~\citeyearpar{chatterjee2015}, a paper that strongly influenced our research in this direction. 

Showing that the risk $\Es_{\overline{s}}\left[h^{2}(\overline{s},\widehat{s})\right]$ at some particular points $\overline s$ can be bounded from above by some quantity $R(\overline{s},n)$ which is of smaller order than the global minimax risk over $\overline S$ requires some specific probabilistic tools that have been established in Baraud~\citeyearpar{Baraud:2014kq}. These tools allow to bound the expectation of the supremum of an empirical process over the neighbourhood of an element $\overline s\in\overline S$ by some quantity which is of smaller order than that one could get by using the global entropy of the class $\overline S$ as, for example, in van de Geer~\citeyearpar{MR1212164}.

The existence of points in the model on which the estimator is superminimax was already noticed for the Grenander estimator of a non-increasing density --- see Grenander~\citeyearpar{MR599175} and Groeneboom~\citeyearpar{MR822052} --- on an interval $[a,+\infty)$ with a known value of $a$. It is shown in Birg\'e~\citeyearpar{MR1026298} that the $\IL_{1}$-risk of the Grenander estimator of a non-increasing piecewise constant density based on at most $D$ intervals is bounded by $c\sqrt{D/n}$, for some positive universal constant $c$, and can therefore be of smaller order than the typical risk for non-increasing densities which is of order  $n^{-1/3}$. We shall see below that, for the same estimation problem, the $\rho$-estimator will perform similarly (up to possible logarithmic factors) with the same superminimaxity property on piecewise constant densities. Moreover, the $\rho$-estimator does not need to know $a$ and is robust with respect to the Hellinger distance. 

The case of monotone densities on $[a,+\infty)$ is far from unique. There are many other examples of families $\overline S$ of densities for which one can find a subset $V$ of $\overline S$ on which the rates of convergence of the $\rho$-estimator are faster than the rate at a ``typical" point of $\overline S$. Moreover, it happens that the set $V$ often possesses good approximation properties with respect to the much larger space $\overline S$. These approximation properties combined with the robustness of $\rho$-estimators as expressed by \eref{eq-robust3} allow to derive non-asymptotic minimax risk bounds over large  subsets of $\overline S$. Such sets are possibly non-compact and therefore neither possess a finite metric dimension nor a finite entropy. 

In view of illustrating this superminimaxity phenomenon, we shall consider in the present paper 
models of densities $\overline S$ defined by some shape constraints, namely piecewise monotone, piecewise convex or concave and log-concave densities. There is a large amount of literature dealing with these density models and we shall content ourselves to mention a few references only and refer the reader to the bibliography therein. For monotone densities we refer to the books by Groeneboom and Wellner~\citeyearpar{MR1180321} and van de Geer~\citeyearpar{MR1739079}. For the estimation of a convex density, we mention Groeneboom {\it et al.}~\citeyearpar{MR1891742} and refer to the papers of Doss and Wellner~\citeyearpar{Doss-Wellner}, D\"umbgen and Rufibach~\citeyearpar{MR2546798}, Cule and Samworth~\citeyearpar{MR2645484} for the estimation of a log-concave density. In the regression setting, let us mention Guntuboyina and Sen~\citeyearpar{MR3405621} for estimating a convex regression function and Chatterjee {\it et al.}~\citeyearpar{chatterjee2015} for the isotonic regression. Recently,
Bellec~\citeyearpar{Bellec:2015qf} extended the results of Chatterjee {\it et al.}~\citeyearpar{chatterjee2015} about the properties of least-squares estimators over convex polyhedral cones in the homoscedastic Gaussian regression framework, to general closed convex subsets of $\R^{n}$ from which he also derived some results of superminimaxity in this specific Gaussian framework. In these two papers the results are restricted to convex models. As opposed, convexity does not play any special role in our presentation and the models we shall use here are not necessarily convex which allows us to deal with more general shape constraints like piecewise monotonicity or log-concavity.

The paper is organised as follows. The statistical setting, main notations and conventions as well as a brief reminder of what a $\rho$-estimator is in the density estimation framework can be found in Section~\ref{sect-SS}. The introductory example of the model of monotone densities in Section~\ref{Mon} gives a first flavour of the results we establish all along the paper. The main result can be found in Section~\ref{sect-main} and its applications to different density models (piecewise constant, piecewise monotone, piecewise convex-concave and log-concave densities) are detailed in Section~\ref{sect-appli}. The problem of model selection is addressed in Section~\ref{sect-MS} and Section~\ref{sect-proof} is devoted to the proofs.

\section{The statistical setting}\label{sect-SS}
Let $(\X,\A)$ be a measurable set, $\mu$ a $\sigma$-finite measure on $(\X,\A)$, $\cP_{\!\mu}$ 
the set of all probabilities on $(\X,\A)$ which are absolutely continuous with respect to $\mu$. We shall denote by $\LL$ the set of real-valued functions from $\X$ to $\R$ and $\LL_\mu$ the subset of $\LL$ consisting of those functions $t\ge 0$ satisfying  
$\int t\,d\mu=1$, that is the set of probability densities with respect to $\mu$. An element 
of $\cP_{\!\mu}$ with density $t\in\LL_{\mu}$ will be denoted by $P_t$. We turn $\cP_{\!\mu}$ into a metric space via the Hellinger distance $h$. We recall from Le Cam~(\citeyear{MR0334381} or \citeyear{MR856411}) that the Hellinger distance between 
two elements $P$ and $Q$ of $\cP_{\!\mu}$ is given by
\[
h(P,Q)=\cro{{1\over 2}\int_{\X}\pa{\sqrt{dP/d\mu}-\sqrt{dQ/d\mu}}^{2}d\mu}^{1/2}.
\]
For $P_{t},P_{u}\in\cP_{\!\mu}$ with $t,u\in\LL_{\mu}$ we shall write $h(t,u)$ for $h(P_t,P_u)$.

We observe $n$ i.i.d.\ random variables $X_{1},\ldots,X_{n}$ with values in $(\X,\A)$ and distribution $P_{s}$ for some density $s\in\LL_{\mu}$.  Although $s$ might not be uniquely defined in $\LL_{\mu}$ as a density with respect to $\mu$ of the distribution of the observations, we shall refer to $s$ as ``the" density of $P_{s}$ for simplicity. 
To avoid trivialities, we shall always assume, in the sequel, that $n\ge3$, so that $\log n>1$.
%
%
The estimators $\widehat{s}$ that we shall consider here will be based on {\it models} for $s$ defined as follows.
\begin{defi}
A density model, or a model (for short), is a subset  of $\LL_{\mu}$ for which there exists an at most countable subset $S\subset \overline S$ such that $\{P_{t},\ t\in S\}$ is dense in $\{P_{t},\ t\in \overline S\}$  with respect to the Hellinger distance. We shall then say that $S$ is dense in $\overline S$ and that $\overline S$ is separable with respect to the Hellinger distance. 
\end{defi}
A density model $\overline S$ should be chosen so that the corresponding probability model $\{P_{t},\ t\in \overline S\}$ approximates the true distribution $P_s$ (with respect to the Hellinger distance). The model $\overline S$ may or may not contain $s$. Of course a model $\overline S$ is good only if the distance $h\left(s,\overline{S}\right)$ is not too large, where we set, for $A\subset\LL_\mu$, $h(t,A)=\inf_{u\in A}h(t,u)$. Our aim in this paper is to study the performance of a $\rho$-estimator $\widehat s$ of $s$ built on $\overline S$. The definition and properties of $\rho$-estimators have been described in great details in Baraud~{\it et al.}~\citeyearpar{Baraud:2014qv} and  we only give below a brief account of what a $\rho$-estimator is.

\subsection{What is a $\rho$-estimator?}\label{sect-RhoEst}
In the context of density estimation based on i.i.d.\ variables, which is the one we consider here, a $\rho$-estimator provides a robust (in our sense) and an (almost) rate optimal estimator over a model $\overline S$ of densities in all cases we know. In order to avoid long developments we restrict ourselves to its construction in the specific situations we shall encounter here, namely when the observations are i.i.d.

Let $\psi$ be the increasing function from $[0,+\infty]$ onto $[-1,1]$ defined by
\[
\psi(u)={u-1\over \sqrt{1+u^{2}}}\quad\mbox{for }u\in [0,+\infty)\qquad\mbox{and}\qquad\psi(+\infty)=1.
\]
Given a model $\overline S$ of densities on $(\X,\A,\mu)$ and a countable and dense subset $S$ of $\overline S$, a $\rho$-estimator $\widehat s$ of the density $s$ on $\overline S$ is defined in the following way. For densities $t,t'\in\LL_\mu$ we set
\[
\gT(\gX,t,t')={n\over 2}\cro{h^{2}\pa{t,{t+t'\over 2}}-h^{2}\pa{t',{t+t'\over 2}}}+{1\over \sqrt{2}}\sum_{i=1}^{n}\psi\pa{\sqrt{t'\over t}(X_{i})}
\]
and define $\widehat s$ as any (measurable) element of the closure of the set
\begin{equation}\label{def-ens}
\left\{t\in S\,\left|\,\gup(S,t)\le \inf_{t'\in S} \gup(S,t')+35.7\right.\right\}\quad\mbox{with}\quad\gup(S,t)=\sup_{t'\in S}\gT(\gX,t,t').
\end{equation}
In the calculation of $\gT(\gX,t,t')$, which involves the ratio $t'/t$, we use the convention $0/0=1$ and $a/0=+\infty$ for $a>0$. The constant 35.7 in~\eref{def-ens} has only been chosen for convenience in the calibration of the numerical constants in the original paper Baraud~{\it et al.}~\citeyearpar{Baraud:2014qv} and can be replaced by any positive number. It is clear from the construction that, given a model $\overline S$, there is not a unique $\rho$-estimator on $\overline S$. However, the risk bounds we derived in Baraud~{\it et al.}~\citeyearpar{Baraud:2014qv} are valid for any version of these $\rho$-estimators.

\subsection{Notations, conventions and definitions}\label{sect-conv}
We set $\log_+(x)=\max\{\log x,1\}$, $\N^*=\N\setminus\{0\}$, $a\vee b=\max(a,b)$, $a\wedge b=\min(a,b)$
and, for $x\in \R_{+}$, $\lceil x\rceil=\inf\{n\in\N,\ n\ge x\}$; $|A|$ denotes the cardinality of the finite set $A$ and $C,C',\ldots$ numerical constants that may vary from line to line.  For a function $f$ on $\R$, $f(x+)$ and $f(x-)$ denote respectively the right-hand and left-hand limits of $f$ at $x$ 
whenever these limits exist. We shall also use the following conventions: $\sum_{\varnothing}=0$, $x/0=+\infty$ if $x>0$, $x/0=-\infty$ if $x<0$ and $0/0=1$. 
%
\begin{defi}\label{D-part}
A partition of the open interval $(a,b)$ ($-\infty\le a<b\le+\infty$) of size $k+1$ with $k\in\N$ is either $\varnothing$ when $k=0$ or a finite set $\II=\{x_1,\ldots, x_k\}$ of real numbers with $a<x_1<x_2<\cdots<x_k<b$ if $k\ge1$. We shall call {\em endpoints} of the partition $\II$ the numbers $x_j$, $1\le j\le k$, and {\em intervals} of the partition the {\em open} intervals $I_j=(x_j,x_{j+1})$, 
$0\le j\le k$ with $x_0=a$ and $x_{k+1}=b$. A partition $\II$ will also be identified to the set of its intervals and we shall equally write $\II=\{I_0,\ldots, I_k\}$ or $\II=\{x_1,\ldots, x_k\}$. 
\end{defi}
The set of all partitions of $\R$ with $k$ endpoints or $k+1$ intervals is denoted by $\JJ(k+1)$ and the length of $I_j$ by $\ell(I_j)$. If $\II=\{x_1,\ldots, x_k\}$ and $\II'=\{x'_1,\ldots, x'_{k'}\}$, $\II\vee\II'=\{x_1,\ldots, x_k\}\cup\{x'_1,\ldots, x'_{k'}\}$ and $\II\succeq\II'$ means that $\{x_1,\ldots, x_k\}\supset\{x'_1,\ldots, x'_{k'}\}$.

\section{Monotone densities}\label{Mon}
In view of illustrating the main result of this paper to be presented in Section~\ref{sect-main}, let us consider the example of the model $\overline S$ consisting of all the densities with respect to the Lebesgue measure $\mu$ that are non-increasing on some arbitrary interval of $\R$ which is open on its left end and vanish elsewhere. In this case $(\X,\A)=(\R,\B(\R))$ and $\overline S$ is the set of all densities of the form $t=f\1_{(\overline{x},+\infty)}$ with $\overline{x}\in\R$ and $f$ is a non-increasing and non-negative function on $(\overline{x},+\infty)$ (which may be unbounded in the neighbourhood of $\overline{x}$) such that $\int_{\overline{x}}^{+\infty}f(x)\,dx=1$. The results we get below would be similar for the set of all densities which are non-decreasing on some interval of $\R$ and vanish elsewhere. 

For $D\in\N^*$ we define $\overline V(D)$ to be the set of all densities of the form $\sum_{j=1}^Da_j\1_{(x_j,x_{j+1}]}$ with $\II=\{x_1,\ldots, x_{D+1}\}\in\JJ(D+2)$ and $a_j\ge0$ for $1\le j\le D$. Note that the densities in $\overline V(D)$ take the value 0 on the two unbounded extremal intervals $I_0$ and $I_{D+1}$ of the partition $\II$. For instance, $\overline V(1)$ corresponds to the family of uniform densities on intervals, that is 
\[
\overline V(1)=\left\{t(\cdot)=\theta_{1}^{-1}\1_{(0,\theta_1]}(\cdot-\theta_{0}),\ \theta_{1}>0,\ \theta_{0}\in\R\right\}.
\] 
In such a situation, we can prove the following result.
\begin{thm}\label{thm-mono}
Any $\rho$-estimator $\widehat s$ on $\overline S$, as defined in Section~\ref{sect-RhoEst}, satisfies
\begin{equation}\label{eq-1}
C\E\!\cro{h^{2}(s,\widehat s)}\le\inf_{D\ge1}\cro{h^{2}\!\left(s,\overline V(D)\cap\overline S\right)
+{D\over n}\log_+^{3}\!\pa{n\over D}}\quad\mbox{for all }s\in\LL_{\mu}
\end{equation}
and some universal constant $C\in(0,1]$.
\end{thm}
%
\begin{remark}
Since $C\le1$, the left-hand side is always bounded by one so that it is useless to consider
values of $D$ that lead to a bound which is not smaller than one, in particular $D\ge n$,
and (\ref{eq-1}) is actually equivalent to
\[
C\E\!\cro{h^{2}(s,\widehat s)}\le\inf_{1\le D<n}\cro{h^{2}\!\left(s,\overline V(D)\cap\overline S\right)
+{D\over n}\log_+^{3}\!\pa{n\over D}}.
\]
Although we shall not repeat it systematically, the same remark will hold for all our subsequent results.
\end{remark}
Bound (\ref{eq-1}) means that the risk function $s\mapsto \E\!\cro{h^{2}(s,\widehat s)}$ of 
$\widehat s$ over $\LL_{\mu}$ can  be quite small in the neighbourhood of some specific densities 
$t\in\overline S$: if $s$ belongs to $\overline V(D)\cap\overline S$ with $D<n$ or is close enough to some density $t\in\overline V(D)\cap\overline S$, the risk of 
$\widehat s$ is of order $D/n$, up to logarithmic factors. More precisely, 
\[
\sup_{s\in \overline V(D)\cap\overline S}\E\!\cro{h^{2}(s,\widehat s)}\le C'{D\over n}\log_+^{3}\!
\pa{n\over D}\quad\mbox{for}\quad1\le D<n.
\]
When $n$ becomes large and $D$ remains fixed, the rate of convergence of $\widehat s$ towards an element of $\overline V(D)\cap\overline S$ is therefore almost parametric.

Of particular interest are the densities $t$ which are bounded, supported on a compact
interval $[a,b]$ of $\R$ (for numbers $a<b$ depending on $t$) and non-increasing on
$(a,b)$. Given $\overline{M}\ge0$, we introduce the set $\overline S(\overline{M})$ of 
densities $t$ of this form and for which 
\begin{equation}
(b-a)V^2_{[a,b]}\left(\sqrt{t}\right)=M(t)\le \overline{M},
\label{Eq-Mt}
\end{equation}
where the {\em variation} $V_{[a,b]}\left(\sqrt{t}\right)$ of the non-increasing function $\sqrt{t}$ on $[a,b]$ is defined
in the following way:
\begin{defi}\label{def-var}
Let the function $f$ be defined on some interval $I$ (with positive length) of $\R$ and monotone on the interior $\mathring{I}$ of $I$. Its variation on $I$ is given by 
\begin{equation}
V_{I}(f)=\sup_{x\in\mathring{I}}f(x)-\inf_{x\in\mathring{I}}f(x)\in[0,+\infty].
\label{Eq-var}
\end{equation}

\end{defi}

Note that $\overline S(0)$ is the set of uniform densities on intervals, so that 
$\overline S(0)=\overline V(1)$, and that $\overline S(\overline{M})$ is not compact and contains densities that can be arbitrarily large in sup-norm. The functional $M$ remains invariant by translation and scaling: if $u(\cdot)=\lambda t\!\left(\lambda(\cdot-\tau)\str{3.5}\right)$ with $\lambda>0$ and $\tau\in\R$, then $M(u)=M(t)$ which implies that $\overline S(\overline{M})$ is also invariant by translation and scaling. It turns out that the densities lying in $\overline S(\overline{M})$ can be well approximated by elements of $\overline V(D)$. More precisely, the following approximation result holds.
\begin{prop}\label{prop-approx}
For all $D\in\N^*$ and $t\in\bigcup_{\overline{M}\ge0}\overline S(\overline{M})$,
\[
h^2\!\left(t,\overline V(D)\cap\overline S\right)\le\left[M(t)/(2D)^2\right]\wedge1.
\]  
\end{prop}
Using the triangle inequality, the right-hand side of~\eref{eq-1} can be bounded from above in the following way: for all $\overline{M}\ge0$ and $t\in \overline S(\overline{M})$,
\begin{eqnarray*}
\lefteqn{ \inf_{D\ge1}\cro{h^{2}\!\left(s,\overline V(D)\cap\overline S\right)+{D\over n}\log_+^{3}\!\pa{n\over D}}}\hspace{35mm}\\&\le& 2h^{2}(s,t)+\inf_{D\ge1}\cro{2h^{2}\!\left(t,\overline V(D)\cap\overline S\right)+{D\over n}\log_+^{3}\!\pa{n\over D}}\\&\le& 2h^{2}(s,t)+\inf_{D\ge1}\cro{\frac{\overline{M}}{2D^2}+{D\over n}\log_+^{3}\!\pa{n\over D}}.
\end{eqnarray*}
Finally, since $t$ is arbitrary in $\overline S(\overline{M})$,
\[
\inf_{D\ge1}\cro{h^{2}\!\left(s,\overline V(D)\cap\overline S\right)+{D\over n}\log_+^{3}\!\pa{n\over D}}\le 2h^{2}\!\left(s,\overline S(\overline{M})\right)+\inf_{D\ge1}\cro{\frac{\overline{M}}{2D^2}+{D\over n}\log_+^{3}\!\pa{n\over D}}.
\]
Optimizing the right-hand side with respect to $D$ and using the facts that $\overline{M}$ is arbitrary and $\log n>1$, we derive the following corollary of Theorem~\ref{thm-mono}.
\begin{cor}\label{cor-decrease}
For all probabilities $P_s$ in $\cP_{\!\mu}$, any $\rho$-estimator $\widehat s$ of $s$ on $\overline S$ 
satisfies, for some constant $C\in(0,1]$,
\begin{equation}\label{res-2}
C\E\!\cro{h^{2}(s,\widehat s)}\le\inf_{\overline{M}\ge0}\cro{h^{2}\!\left(s,\overline S(\overline{M})\right)+\left(\left(\overline{M}^{1/3}n^{-2/3}(\log n)^{2}\right)\bigvee \left(n^{-1}(\log n)^{3}\right)\right)}.
\end{equation}
\end{cor}

In particular, if $s\in\overline S(\overline{M})$ for some $\overline{M}\ge n^{-1}(\log n)^{3}$, the risk bound of the estimator is not larger (up to a universal constant) than $\overline{M}^{1/3}n^{-2/3}(\log n)^{2}$ while for smaller values of $\overline{M}$ it is bounded by $n^{-1}(\log n)^{3}$. Up to logarithmic factors, this rate
(with respect to $n$) is optimal since it corresponds to the lower bound of order $n^{-2/3}$ for
the minimax risk on the subset of $\overline S(\overline{M})$ consisting of the non-increasing densities 
supported in $[0,1]$ and bounded by $\overline{M}$. This lower bound follows from the proof of Proposition~1 of Birg\'e~\citeyearpar{MR902241}. 
The result was actually stated in this paper for the $\IL_1$-distance 
but its proof shows that it applies to the Hellinger distance as well. This property means that, although the set $\overline S(\overline{M})$ is not compact because the support of the densities is unknown, the minimax risk on $\overline S(\overline{M})$ is finite. We do not know any other estimator with the same performance which is also robust with respect to Hellinger deviations.

Note that Corollary~\ref{cor-decrease} can also be used to determine the rate of estimation for
decreasing densities $s$ with possibly unbounded support and maximum value, provided that
we have some assumption about the behaviour of the function $\overline{M}\mapsto h\!\left(s,\overline S(\overline{M})\right)$ when $\overline{M}$ goes to infinity.

\section{The main result}\label{sect-main}
Let us start with some definitions.
\begin{defi}
A class $\CC$ of subsets of $\X$ is said to shatter a finite subset $A=\{x_{1},\ldots,x_{m}\}$
of $\X$ if the class of subsets
\begin{equation}
\CC\cap A=\{C\cap A,\ C\in\CC\}
\label{Eq-CA}
\end{equation}
is equal to the class of all subsets of $A$ or, equivalently, if $|\CC\cap A|=2^{m}$. A non-empty class 
$\CC$ of 
subsets of $\X$ is a VC-class with dimension $d\in\N$ if there exists some integer $m$
such that no finite subset $A\subset \X$ with cardinality $m$ can be shattered
by $\CC$ and $d+1$ is the smallest $m$ with this property.
\end{defi}
\begin{defi}\label{def-wVC}
Let $\FF$ be a non-empty class of functions on a set $\X$ with values in $[-\infty,+\infty]$. We shall say that $\FF$ is weak VC-major with dimension $d\in\N$ if $d$ is the smallest integer $k\in\N$ such that, for all $u\in\R$, the class
\begin{equation}
\CC_{u}(\FF)=\left\{\strut\{x\in\X,\ \ f(x)>u\},\ f\in\FF\right\}
\label{Eq-Cu}
\end{equation}
is a VC-class of subsets of $\X$ with dimension not larger than $k$.  
\end{defi}
We may now introduce the main property to be used in this paper.
\begin{defi}\label{def-extremal}
Let $\FF$ be a class of real-valued functions on $\X$. We shall say that an element $\overline f\in \FF$ is extremal in $\FF$ (or is an extremal point of $\FF$) with degree $d(\overline f)=d\vee 1\in\N^*$ if the class of functions
\[
(\FF/\overline f)=\{{f/\overline f},\ f\in\FF\},
\]
is weak VC-major with dimension $d$.
\end{defi}
\begin{prop}\label{prop-fonda}
Let $\FF$ be a class of nonnegative functions on $\X$. The element $\overline f\in \FF$ is extremal in $\FF$ with degree not larger than $2d$ if for all $\lambda\ge 0$, 
\[
\CC(\FF,\overline f,\lambda)=\ac{\X}\cup\left\{\strut\{x\in\X\,|\,f(x)-\lambda\overline f(x)>0\},
\; f\in \FF\right\}
\]
is a VC-class with dimension not larger than $d\ge 1$. 
\end{prop}
\begin{proof}
Let us bound the VC dimension of $\CC_{u}((\FF/\overline f))$ according to the value of $u\in\R$. If $u<0$, $\CC_{u}((\FF/\overline f))=\{\X\}$ and is therefore VC with dimension not larger than $0\le 2d$. Let us now assume that $u\ge 0$ and set $A=\{\overline f>0\}$. Using Lemma~\ref{L-trans} (in Section~\ref{prelim}), it suffices to prove that $\CC_{u}((\FF/\overline f))\cap A$ and $\CC_{u}((\FF/\overline f))\cap A^{c}$ are two VC-classes with dimensions not larger than $d$. For all $x\in A$ and $f\in\FF$
\[	
f(x)/\overline{f}(x)>u\qquad\mbox{is equivalent to}\qquad f(x)>u\overline{f}(x)
\]
showing thus that $\CC_{u}((\FF/\overline f))\cap A\subset \CC(\FF,\overline f,u)\cap A$ and is therefore VC with dimension not larger than $d$. Let us now turn to the case where $x\not\in A$, which means that $\overline f(x)=0$ so that $f(x)/\overline f(x)$ is either 1 or $+\infty$ (with our conventions). For $u\ge 1$ and all $x\in A^{c}$,
\[
f(x)/\overline{f}(x)>u\qquad\mbox{is equivalent to}\qquad f(x)>0
\]
and $\CC_{u}((\FF/\overline f))\cap A^{c}\subset \CC(\FF,\overline f,0)\cap A^{c}$ and is therefore VC with dimension not larger than $d$.
For $u\in [0,1)$, $(f/\overline{f})(x)>u$ for all $x\in A^{c}$, hence $\CC_{u}((\FF/\overline f))\cap A^{c}=\{A^{c}\}$ which is VC with dimension $0<d$ and this concludes the proof.
\end{proof}

Let us now state our main result.
\begin{thm}\label{thm-main}
Let $\overline S$ be a model with a non-void set $\overline \Lambda$ of extremal points. Any $\rho$-estimator $\widehat{s}$ on $\overline S$ satisfies, for some universal 
constant $C\in(0,1]$,
\[
\P_s\!\cro{Ch^{2}(s,\widehat s)\le\inf_{\overline s\in\overline \Lambda}
\cro{h^{2}(s,\overline s)+{d(\overline s)\over n}\log_+^{3}\!\pa{n\over d(\overline s)}}+{\xi\over n}}
\ge 1-e^{-\xi}\quad\mbox{for all }\xi>0,
\]
whatever the true distribution $P_s\in \cP_{\!\mu}$. Consequently, 
\[
C\E\!\cro{h^{2}(s,\widehat s)}\le\inf_{\overline s\in\overline \Lambda}
\cro{h^{2}(s,\overline s)+{d(\overline s)\over n}\log_+^{3}\!\pa{n\over d(\overline s)}}.
\]
\end{thm}
Note that the boundedness of $h$ implies that values of $d(\overline s)\ge n$ lead to a trivial bound so that the infimum could be reduced to those $\overline{s}$ such that $d(\overline s)<n$. We do not know to what extend the $\log^{3}$ factor is necessary. We believe that it is not optimal although a log-factor appears to be necessary in some situations as shown by the example of Section~\ref{histo} below.

\section{Applications}\label{sect-appli}
Throughout this section, $(\X,\A)=(\R,\B(\R))$ and $\mu$ is the Lebesgue measure on $\R$. In particular, we shall only consider densities with respect to the Lebesgue measure. We start with the following useful lemma:
\begin{lem}\label{lem-vcm}
If $\D$ is a class of subsets of $\R$ such that each element of $\D$ is the union of at most $k$ intervals, $\D$ is VC with dimension at most $2k$.
\end{lem}
\begin{proof}
Let $x_{1}<x_2<\ldots<x_{2k+1}$ be $2k+1$ points of $\R$. It is easy to check that elements of the form $J_{1}\cup\dots\cup J_{l}\in\D$, where the $J_{j}$ are disjointed intervals and $l\le k$, cannot pick up the subset of points $\bigcup_{i=0}^{k}\{x_{2i+1}\}$.
\end{proof}
%

\subsection{Piecewise constant densities\label{histo}}
Let us now consider the model $\overline{V}(D)$ of Section~\ref{Mon} to build a $\rho$-estimator. If 
$f$ and $\overline{f}$ belong to $\overline{V}(D)$, for all $\lambda\ge0$, $f-\lambda\overline{f}$ is of the form $\sum_{j=1}^{k}a_j\1_{(x_j,x_{j+1}]}$ with $k<2(D+1)$ so that $\{x\in\X\,|\,f(x)-\lambda\overline f(x)>0\}$ is the union of at most $D+1$ disjointed intervals. Applying Lemma~\ref{lem-vcm} and Proposition~\ref{prop-fonda} to the sets $\D=\CC(\overline V(D),\overline f,\lambda)$ with $\lambda\ge 0$, we obtain that all the elements of $\overline{V}(D)$ are extremal in $\overline{V}(D)$ and their degrees are not larger than $4(D+1)$. We therefore deduce from Theorem~\ref{thm-main} that
\begin{equation}
\sup_{s\in\overline{V}(D)}\E\!\cro{h^{2}(s,\widehat s)}\le C{D\over n}\log_+^{3}\!\pa{n\over D},
\label{Eq-histo}
\end{equation}
which, up to the logarithmic factor, corresponds to a parametric rate (with respect to $n$) although the partition that defines $s$ can be arbitrary in $\JJ(D+2)$ and the support of
$s$ is unknown. It follows from Birg\'e and Massart~\citeyearpar{MR1653272}, Proposition~2,
that a lower bound for the minimax risk on $\overline{V}(D)$ is of the form $C'(D/n)\log_+(n/D)$, which shows that some power of $\log_+(n/D)$ is necessary in (\ref{Eq-histo}). We suspect that the power three for the logarithm is not optimal.

\subsection{Piecewise monotone densities\label{Sect6}}
Let us now see how Theorem~\ref{thm-main} can be applied in the simple situation of piecewise monotone
densities.
\begin{defi}
Given $k\in\N^*$ and a partition $\II=\{I_0,\ldots, I_{k-1}\}\in\JJ(k)$, a real-valued function $f$
on $\R$ will be called piecewise monotone (with $k$ pieces) based on $\II$ if $f$ is monotone on each open interval $I_j$, $0\le j\le k-1$. The set of all such functions will be denoted by $\GG_k$. For $k\ge2$ (since no density is monotone on $\R$), $\FF_k$ is the set of densities (with respect to the Lebesgue measure) that belong to $\GG_k$.  
\end{defi}
Clearly, $\GG_k\subset\GG_l$ and $\FF_k\subset\FF_l$ for all $l>k$. 
\begin{prop}\label{prop-extremal1}
For $D\ge1$ and $k\ge2$, any element $\overline{f}$ of $\FF_k\cap\overline{V}(D)$ is extremal in $\FF_k$ with degree not larger than $3(k+D+1)$. 
\end{prop}
\begin{proof}
Let $f$ be a piecewise monotone density on $\R$ based on a partition $\II_0\in\JJ(k)$, therefore with $k-1$ endpoints. Let $\overline{f}\in\overline{V}(D)\cap\FF_{k}$ be a piecewise constant density based on a partition $\II_1\in\JJ(D+2)$ (with $D+1$ endpoints) and let $\II_2=\II_1\vee\II_0$. It is a partition of $\R$ with at most $k+D$ endpoints, therefore at most $k+D+1$ intervals and on each such interval $f$ is monotone and $\overline{f}$ is constant which implies that $f-\lambda\overline{f}$ belongs to $\GG_{k+D+1}$ for all $\lambda\ge0$. It then follows from Lemma~\ref{lem-km} below that the sets $\{x\in\X\,|\,f(x)-\lambda\overline f(x)>0\}$ are unions of at most $(3/2)(k+D+1)$ intervals. The conclusion follows from Lemma~\ref{lem-vcm} and Proposition~\ref{prop-fonda} applied to $\D=\CC(\FF_{k},\overline f,\lambda)$ with $\lambda\ge 0$. 
\end{proof}
\begin{lem}\label{lem-km}
If $f\in\GG_k$, whatever $a\in\R$ the set $\{x\in\X\,|\,f(x)>a\}$ can be written as a
union of at most $k+\lceil(k-1)/2\rceil\le3k/2$ intervals and the set $\{x\in\X\,|\,f(x)\le a\}$ as well.
\end{lem}
\begin{proof}
Let $f\in\GG_{k}$, $\II$ be the partition of $\R$ with $k$ open intervals associated to $f$ and $x_{1},\ldots,x_{k-1}$ the $k-1$ endpoints of this partition. For $I_j\in\II$, $\{x\in\X\,|\,f(x)>a\}\cap I_j$ is either 
$\varnothing$ or a non-void interval and
\begin{equation}\label{eq-decompI}
\{x\in\X\,|\,f(x)>a\}=\left(\bigcup_{j=1}^{k}\cro{\{f>a\}\cap I_{j}}\right)\bigcup\left(\bigcup_{j=1}^{k-1}\cro{\{x_{j}\}\cap\{f>a\}}\right).
\end{equation}
This decomposition shows that $\{x\in\X\,|\,f(x)>a\}$ is the union of at most $2k-1$ disjointed intervals. Nevertheless, this bound can be refined as follows. If $f(x_{j})\le a$, $\{x_{j}\}\cap\{f>a\}=\varnothing$ and the only situation we need to consider is when $f(x_{j})>a$ in which case $\{x_{j}\}\cap\{f>a\}=\{x_{j}\}$. If $x_{j}$ belongs to the closure of one of the intervals of the form $\{f>a\}\cap I_{j'}$,  the set $\cro{\{f>a\}\cap I_{j'}}\cup\{x_{j}\}$ only counts for one interval in~\eref{eq-decompI}. The only situation for which $\{x_{j}\}$ adds an extra interval occurs when $f(x_{j-1}+)>a$, $f(x_{j}-)\le a$, $f(x_{j})>a$, $f(x_{j}+)\le a$ and $f(x_{j+1}-)>a$. The number of such points $x_{j}$ is not larger than $\lceil(k-1)/2\rceil$ and $\{x\in\X\,|\,f(x)>a\}$  is therefore the union of at most $k+\lceil(k-1)/2\rceil$ intervals. The proof for $\{x\,|\,f(x)\le a\}$ is the same.
\end{proof}
An application of Theorem~\ref{thm-main} with $\overline \Lambda=\FF_{k}\bigcap\cro{\bigcup_{D\ge 1}\overline{V}(D)}$, the elements of which are extremal in $\FF_{k}$ by Proposition~\ref{prop-extremal1}, leads to the following result.
\begin{cor}\label{cor-Fk}
For all $k\ge2$, any $\rho$-estimator on $\FF_k$ satisfies, for all distributions $P_s\in\cP_{\!\mu}$,
\[
C\E\!\cro{h^{2}(s,\widehat s)}\le\inf_{D\ge1}\cro{\inf_{\overline s\in\FF_k\bigcap\overline{V}(D)}
h^{2}(s,\overline s)+{k+D\over n}\log_+^{3}\!\pa{n\over k+D}}
\]
where $C\in(0,1]$ is a universal constant.
\end{cor}
Note that the bound is trivial for $k\ge n-1$ and that using $D$ with $k+D\ge n$ also leads to a trivial bound so that we should restrict ourselves to $D<n-k$ when $k<n-1$.

Since, for $t\in\FF_k$,
\[
\inf_{\overline s\in\FF_k\bigcap\overline{V}(D)}h(s,\overline s)\le\inf_{t\in\FF_k} \left[h(s,t)+\inf_{\overline s\in\FF_k\bigcap\overline{V}(D)}h(t,\overline s)\right],
\]
to go further with our analysis it will be necessary to evaluate $\inf_{\overline s\in\FF_k\bigcap\overline{V}(D)}h(t,\overline s)$ for $t\in\FF_k$.
In order to do this we shall use an approximation result based on the following functional $M_k$.
\begin{defi}
Let $t\in\FF_k$ for $k\ge2$ and $\II=\{I_0,\ldots, I_{k-1}\}\in\JJ(k)$ a partition on which $t$ is based. Using the convention $(+\infty)\times0=0$, we define 
\begin{equation}
M_k(t,\II)=\left[\sum_{j=0}^{k-1}\left[\ell(I_j)V_{I_j}^2\left(\sqrt{t}\right)\right]^{1/3}\right]^3\le+\infty,
\label{Eq-Mk}
\end{equation}
where $V_{I_j}\left(\sqrt{t}\right)$ is the variation of $\sqrt{t}$ on $I_j$ given by
(\ref{Eq-var}). The functional $M_{k}$ is defined on $\FF_{k}$ as 
\[
M_k(t)=\inf_{\II}M_k(t,\II)\quad\mbox{for all }t\in\FF_{k},
\]
where the infimum runs among all partitions $\II\in\JJ(k)$ on which $t$ is based. For $0\le \overline{M}<+\infty$ and $k\in\N^*$, we denote by $\FF_{k+2}(\overline{M})$ the subset of $\FF_{k+2}$ of those densities $t$ such that $M_{k+2}(t)\le\overline{M}$.
\end{defi}
Note that with our convention, if $M_k(t,\II)<+\infty$, $t$ is equal to zero on $I_0\cup I_{k-1}$, in which case the summation in (\ref{Eq-Mk}) can be restricted to $1\le j\le k-2$, which requires that $k>2$, and that $\FF_{k+2}(0)$ is equal to $\overline{V}(k)$ (in the $\Bbb{L}_1$ sense). The functional $M_k$ is translation and scale invariant which means that it takes the same value at $t$ and $\lambda^{-1}t((\cdot-\tau)/\lambda)$ whatever $\lambda>0$ and $\tau\in\R$. Besides, it possesses the following property.
\begin{lem}\label{lem-comp}
For all $l>k$ and $t\in\FF_{k}\subset \FF_{l}$, $M_l(t)\le M_k(t)$.
\end{lem}
\begin{proof}
Let $\II\in\JJ(k)$ on which $t$ is based. For all partitions $\II'\in\JJ(l)$ satisfying $\II' \succeq \II$, $t$ can be viewed as an element of $\FF_{l}$ based on $\II'$ and consequently it suffices to show that $M_l(t,\II')\le M_k(t,\II)$. In fact,  it suffices to show that, when we simply divide an interval $J$ of length $L$ of $\II$ into $m$ intervals $J_1,\ldots,J_m$ of respective lengths $L_1,\ldots, L_m$,
\[
\sum_{j=1}^{m}\left[L_{j}V_{J_j}^2\left(\sqrt{t}\right)\right]^{1/3}\le L V_{J}^2\left(\sqrt{t}\right)\ \ \mbox{when}\ \ \sum_{j=1}^{m}L_j=L\quad\mbox{and}\quad\sum_{j=1}^{m}V_{J_j}\left(\sqrt{t}\right)\le
V_J\left(\sqrt{t}\right).
\]
Setting $L_j=\alpha_jL$ and $V_{J_j}\left(\sqrt{t}\right)=\beta_jV_J\left(\sqrt{t}\right)$, this amounts to show that $\sum_{j=1}^{m}\alpha_j^{1/3}\beta_j^{2/3}\le1$ which follows from H\"older's Inequality.
\end{proof}
The approximation of elements of $\FF_{k+2}(\overline{M})$ by elements of $\overline{V}(D)$
is controlled in the following way.
%
\begin{prop}\label{P-approkM}
Let $k\ge1$ and $t\in\FF_{k+2}$ with $M_{k+2}(t)<+\infty$. Then, for all $D\ge1$,
\[
h^2\!\left(t,\overline{V}(D+k)\cap\FF_{k+2}\right)\le(2D)^{-2}M_{k+2}(t).
\]
\end{prop}
 Applying Corollary~\ref{cor-Fk} leads to the following bound which is valid for all $t\in\FF_{k+2}$ with $M_{k+2}(t)<+\infty$ and whatever the distribution $P_s$ of the observations:
\begin{equation}
C\E\!\cro{h^{2}(s,\widehat s)}\le h^2(s,t)+\inf_{D\ge1}
\cro{\frac{M_{k+2}(t)}{D^2}+{k+D\over n}\log_+^{3}\!\pa{n\over k+D}}.
\label{Eq-22}
\end{equation}
A final optimization with respect to $D$ leads to
\[
C\E\!\cro{h^{2}(s,\widehat s)}\le h^2(s,t)+\left[M_{k+2}(t)\right]^{1/3}n^{-2/3}(\log n)^{2}+kn^{-1}(\log n)^{3}.
\]
Since this result is valid for all densities $t\in\FF_{k+2}$, we can again optimize it with respect to $t$
which finally leads to:
\begin{thm}\label{thm-kmon}
Any $\rho$-estimator $\widehat{s}$ based on the model $\FF_{k+2}$ for some $k\ge1$ satisfies
\[
C\E\!\cro{h^{2}(s,\widehat s)}\le\inf_{\overline{M}>0}\left[h^2\!\left(s,\FF_{k+2}(\overline{M})\right)+(\log n)^{2}\left(\left(\overline{M}n^{-2}\right)^{1/3}\bigvee\left(kn^{-1}\log n\right)\right)\right],
\]
 for all distributions $P_s\in\cP_{\!\mu}$. In particular
\[
\sup_{s\in\FF_{k+2}(\overline{M})}\E\!\cro{h^{2}(s,\widehat s)}\le C(\log n)^{2}\left[\left(\overline{M}n^{-2}\right)^{1/3}\bigvee\left(kn^{-1}\log n\right)\right].
\]
\end{thm}
If we want to estimate a bounded unimodal density $s$ with support of finite length $L$, we may build a $\rho$-estimator on $\FF_4$. In such a case, $M_4(s)$ can be bounded by 
$4L\|s\|_\infty\ge4$ (since $s$ is a density, $L\|s\|_\infty\ge1$) and the performance of the $\rho$-estimator for such a unimodal density $s$ will be given by
\[
\E\!\cro{h^{2}(s,\widehat s)}\le C\left(L\|s\|_\infty n^{-2}\right)^{1/3}(\log n)^{2}.
\]
%
\subsection{Piecewise concave-convex densities\label{conc}}
In the previous sections we considered densities $t$ which were piecewise monotone or constant which implied the same properties for $\sqrt{t}$ but it follows from Proposition~\ref{P-approkM} that it is actually the approximation properties of $\sqrt{t}$ that matter. This derives from the fact that the Hellinger distance is an $\IL_{2}$-distance between the square roots of the densities. When going to more sophisticated properties than monotonicity, it is no more the same to state them for $t$ or for $\sqrt{t}$ which accounts for the slightly more complicated structure of this section.
%
\begin{defi}\label{D-Conc}
Let $\II\in\JJ(k)$ be a partition with $k$ intervals. A function $f$ is piecewise convex-concave based on $\II$ if it is either convex or concave on each (open) interval $I_j$ of the partition. The set of all such functions when $\II$ varies in $\JJ(k)$ will be denoted by $\GG^1_k$. For $D\in\N^*$ we denote by $\overline W_1(D)$ the set of all functions $\gamma$ of the form $\gamma=\sum_{j=1}^D\gamma_j\1_{(x_j,x_{j+1}]}$ with $x_1<x_2<\cdots<x_{D+1}$ where $\gamma_j$ is an affine function for all $j$. The sets 
$\FF^1_k$ and $\overline V_1(D)$ are the sets of those densities $t$ such that $\sqrt{t}$ belongs to $\GG^1_k$ and $\overline W_1(D)$ respectively.
\end{defi}
We recall that if $f$ is either concave or convex on some open interval $I$, it is continuous on $I$ and admits on $I$ a right-hand derivative $f'$ which is monotone.

The following result will prove useful to find extremal points of $\FF_k^1$.
\begin{lem}\label{lem-14}
For all $k\in\N^*$, $\GG^1_k\subset \GG_{2k}$.
\end{lem}
\begin{proof}
Since $f\in \GG^1_k$ there exists $\II\in\JJ(k)$ such that $f$ is either convex or concave on each open interval $I_j$ of $\II$. The right-hand derivative $f'$ of $f$ on $I_{j}$ being monotone, the sets $\{x\in I_{j},\ f'\le 0\}$ and $\{x\in I_{j},\ f_{j}'>0\}$ are two disjointed subintervals of $I_{j}$ on which $f$ is monotone.
\end{proof}
\begin{prop}\label{prop-extremal2}
For all $D,k\in\N^*$, the elements $\overline f\in\FF^1_k\cap\overline V_1(D)$ are extremal in $\FF^1_k$ with degrees not larger than $12(D+k+1)$.
\end{prop}
\begin{proof}
Let us consider $g\in\GG^1_k$ and $\overline g\in \GG^1_k\cap \overline{W}_1(D)$. There exists a partition $\II_0$ with $k-1$ endpoints such that $g$ is either convex or concave on each interval of $\II_0$ and a partition $\II_{1}$ with $D+1$ endpoints such that $\overline g$ is affine on each interval of $\II_{1}$. The partition $\II_0\vee\II_1$ contains at most $k+D+1$ intervals and on each of these intervals  
$g-\lambda \overline g$ is either convex or concave for all $\lambda\in\R_+$. Hence, the function $g-\lambda \overline g$ belongs to $\GG_{k+D+1}^{1}$ which is a subset of $\GG_{2(k+D+1)}$ by Lemma~\ref{lem-14}. It then follows from Lemma~\ref{lem-km} that $\{g-\lambda \overline g>0\}$ is the union of at most $3(k+D+1)$ intervals. Since $\lambda$ is arbitrary in $\R_{+}$ we conclude with Lemma~\ref{lem-vcm} that $\CC(\GG_{k}^{1},\overline g)$ is VC with dimension not larger than $6(D+k+1)$, which shows by Proposition~\ref{prop-fonda} that the elements $\overline{g}\in\GG^1_k\cap\overline{W}_1(D)$ are extremal in $\GG^1_k$ with degrees not larger than $12(D+k+1)$. The conclusion follows from an application of Lemma~\ref{L-trans} of Section~\ref{prelim} with $\alpha=1/2$.
\end{proof}
 We may now apply Theorem~\ref{thm-main} with $\overline \Lambda=\FF_{k}^{1}\bigcap\cro{\bigcup_{D\ge 1}\overline V_{1}(D)}$ which consists of extremal points of $\FF_{k}^{1}$ and deduce the following risk bound from Proposition~\ref{prop-extremal2}:
\begin{cor}\label{cor-cc}
For all $k\ge2$, any $\rho$-estimator on $\FF_k^1$ satisfies for all distributions $P_s\in\cP_{\!\mu}$,
\[
C\E\!\cro{h^{2}(s,\widehat s)}\le\inf_{D\ge1}\cro{\inf_{\overline s\in\overline V_1(D)\cap\FF^1_k}h^{2}(s,\overline s)+{k+D\over n}\log_+^{3}\!\pa{n\over k+D}}
\]
where $C\in(0,1]$ is a universal constant.
\end{cor}
The control of the approximation term $\inf_{\overline s\in\overline V_1(D)\cap\FF^1_k}h(s,\overline s)$ is analogue to the one we derived in the previous section for $\inf_{\overline s\in\FF_k\bigcap\overline{V}(D)}h(t,\overline s)$ but is based on a new functional:
\begin{defi}
Let $t\in\FF^1_k$ and $\II=\{I_0,\ldots, I_{k-1}\}\in\JJ(k)$ a partition on which $t$ is based, that is, $\sqrt{t}$ is either convex or concave with monotone right-hand derivative $\left(\sqrt{t}\right)'$ on each $I_j$. Using the convention $(+\infty)\times0=0$, we define 
\[
M_{k,1}(t,\II)=\left[\sum_{j=0}^{k-1}\left[[\ell(I_j)]^3V_{I_j}^2\left(\left(\sqrt{t}\right)'\right)\right]^{1/5}\right]^5\le+\infty,
\]
where $V_{I_j}\left(\left(\sqrt{t}\right)'\right)$ is the variation of $\left(\sqrt{t}\right)'$ on $I_j$. The functional $M_{k,1}$ is defined on $\FF^1_k$ as
\[
M_{k,1}(t)=\inf_{\II}M_{k,1}(t,\II)\quad\mbox{for all }t\in\FF^{1}_{k},
\]
where the infimum runs among all partitions $\II\in\JJ(k)$ on which $t$ is based. For $0\le \overline{M}<+\infty$ and $k\in\N^*$, we denote by $\FF^1_{k+2}(\overline{M})$ the subset of $\FF^1_{k+2}$ of those densities $t$ such that $M_{k+2,1}(t)\le\overline{M}$.
\end{defi}
Note that if $M_{k,1}(t,\II)$ is finite, the density $t$ is necessarily zero on the two extremal (unbounded) intervals of the partition $\II$ and therefore $k\ge3$. An analogue of Lemma~\ref{lem-comp} holds for the functional $M_{k,1}(t)$ with a similar proof, saying that if $l>k$ and $t\in\FF^1_{k}\subset \FF^1_{l}$,
then $M_{l,1}(t)\le M_{k,1}(t)$. We omit the details.
%
\begin{prop}\label{P-approcc}
Let $k\ge1$ and $t\in\FF^1_{k+2}$ with $M_{k+2,1}(t)<+\infty$. Then, for $D\ge1$,
\[
h^2\!\left(t,\overline{V}_1\!\left(2(D+k)\strut\right)\cap\FF^1_{k+2}\right)\le(D/2)^{-4}M_{k+2,1}(t).
\]
\end{prop}
Now arguing as we did in the previous section we derive from Corollary~\ref{cor-cc} and Proposition~\ref{P-approcc} our concluding result.
%
\begin{thm}\label{thm-kcc}
Any $\rho$-estimator $\widehat{s}$ based on the model $\FF_{k+2}^1$ with $k\ge1$ satisfies, for all $P_s\in\cP_{\!\mu}$,
\[
C\E\!\cro{h^{2}(s,\widehat s)}\le\inf_{\overline{M}>0}\left[h^2\!\left(s,\FF_{k+2}^1(\overline{M})\right)+\left(\left(\overline{M}n^{-4}\right)^{1/5}(\log n)^{12/5}\right)\bigvee\left(kn^{-1}(\log n)^3\right)\right].
\]
If, in particular, $s\in\FF_{k+2}^1(\overline{M})$, then
\[
\E\!\cro{h^{2}(s,\widehat s)}\le C\left[\left(\left(\overline{M}n^{-4}\right)^{1/5}(\log n)^{12/5}\right)\bigvee\left(kn^{-1}(\log n)^{3}\right)\right].
\]
\end{thm}
%
\subsection{Log-concave densities\label{lconc}}
We now want to investigate a situation which is close to the previous one, the case of log-concave densities on the line. These are densities of the form $\1_I\exp(g)$ for some open interval $I$ of $\R$, possibly of infinite length, and some concave function $g$ on $I$. Let us denote by $\FF'$ the set of all such densities and by $\overline{V}'(D)$ the subset of $\FF'$ of those densities for which $g$ is piecewise affine on $I$ with $D$ pieces. For instance, the exponential density belongs to $\overline{V}'(1)$ while the Laplace density belongs to $\overline{V}'(2)$. Also note that if $\1_I\exp(g)$ is log-concave, the same holds for its square root $\1_I\exp(g/2)$.
\begin{prop}\label{prop-extremal3}
For all $D\in\N^*$, the elements of $\,\overline{V}'(D)$ are extremal in $\FF'$ with degrees not larger than $12(D+2)+4$. 
\end{prop}
\begin{proof}
Let us consider $\1_I\exp(g)\in\FF'$ and $\1_J\exp(\overline{g})\in\overline{V}'(D)$. Then 
the set on which $\1_I\exp(g)>\lambda\1_J\exp(\overline{g})$, with $\lambda\ge0$ is the subset of $I$ on which
\[
g>\log\lambda+\log\1_J+\overline{g},
\]
with the convention that $\log0=-\infty$. If $\lambda=0$ it is the set $I$ itself. Otherwise
it is equal to the union of $I\cap J^c$ and $I\cap J\cap\{g-\overline{g}>\log\lambda\}$. Since on the interval $I\cap J$, $g$ is concave and $\overline g$ piecewise affine with at most $D$ pieces, the function $h=(g-\overline g)\1_{I\cap J}+(\log\lambda)\1_{(I\cap J)^c}$ is piecewise concave on $\R$ with at most $D+2$ pieces. Hence $h$ belongs to $\GG_{D+2}^{1}$ and by Lemma~\ref{lem-14} it also belongs to $\GG_{2(D+2)}$ and it follows from Lemma~\ref{lem-km} that $I\cap J\cap\{g-\overline{g}>\log\lambda\}=\{h>\log \lambda\}$ is the union of at most $3(D+2)$ intervals. Consequently, the set $\{\1_I\exp(g)>\lambda\1_J\exp(\overline{g})\}$ is the union of at most $3(D+2)+1$ intervals and we derive from Lemma~\ref{lem-vcm} that the VC-dimension of $\CC(\FF',\1_J\exp(\overline{g}))$ is not larger than $6(D+2)+2$. The conclusion then follows from Proposition~\ref{prop-fonda}. 
\end{proof}
We may now apply Theorem~\ref{thm-main} with $\overline \Lambda=\bigcup_{D\ge 1}\overline V'(D)$ and use Proposition~\ref{prop-extremal3} to derive the following risk bound.
\begin{cor}\label{cor-lc}
Any $\rho$-estimator on $\FF'$ satisfies, for all distributions $P_s\in\cP_{\!\mu}$,
\[
C\E\!\cro{h^{2}(s,\widehat s)}\le\inf_{D\ge1}\cro{\inf_{\overline s\in\overline V'(D)}h^{2}(s,\overline s)+{D\over n}\log_+^{3}\!\pa{n\over D}}
\]
where $C\in(0,1]$ is a universal constant.
\end{cor}
In particular, if $s\in\overline V'(D)$,
\[
\E\!\cro{h^{2}(s,\widehat s)}\le C{D\over n}\log_+^{3}\!\pa{n\over D}, 
\]
which means that the elements of $\overline V'(D)$ can be estimated by the $\rho$-estimator at a parametric rate, up to some $(\log n)^3$ factor. This is the case for all uniform densities, for exponential densities and their translates and for the Laplace density among many others.
%
\begin{remark}
For simplicity we have restricted our study to log-concave densities but we could as well handle the case of piecewise log-concave densities with several pieces, that is densities of the form $\sum_{j=1}^{k}\1_{I_{j}}\exp(g_{j})$ for concave functions $g_{j}$. The extension would be similar to that which leads from monotone to piecewise monotone and is straightforward. 
\end{remark}
\section{Model selection}\label{sect-MS}
All results of Sections~\ref{histo} to \ref{conc} were based on the use of a single model: $\overline{V}(D)$ in Section~\ref{histo}, $\FF_{k+2}$ in Section~\ref{Sect6} and $\FF_k^1$ in Section~\ref{conc}, which implies that our risk bounds depend on $D$ in the first case and on $k$ in the other cases. In order to get the best possible value of either $D$ or $k$ for the unknown distribution $P_s$, we may use a selection procedure. There are different ways to do this but we shall explain how to do it using Theorem~9 of Birg\'e~\citeyearpar[Section~9]{MR2219712}. To simplify the presentation, we assume that the number $n$ of observations is even with $n=2p$ and we split the sample $\bm{X}=(X_1,\ldots,X_n)$ into two parts of size $p$, $\bm{X}_1$ and $\bm{X}_2$. We also consider all the models $\FF_{j+2}$, $j\ge1$, and $\FF_{k+2}^1$, $k\ge1$, simultaneously. For each of these models we fix a weight $\Delta(j)=j$ and $\Delta(k)=k$. It follows that
\begin{equation}
\sum_{j\ge1}\exp[-\Delta(j)]+\sum_{k\ge1}\exp[-\Delta(k)]=\frac{2}{e-1}.
\label{eq-delta}
\end{equation}
We may now use each of our models to build a $\rho$-estimator based on the sample $\bm{X}_1$. This results in a family of estimators $\widehat{s}_{j}(\bm{X}_1)$, $j\ge1$, and $\widehat{s}_{k}(\bm{X}_1)$, $k\ge1$. The risks of these estimators are bounded according to Theorems~\ref{thm-kmon} and \ref{thm-kcc}. In the second step, we consider these preliminary estimators based on sample $\bm{X}_1$ as a set of points in $\LL_\mu$. We may apply to them the selection procedure described in Section~9.1 of 
Birg\'e~\citeyearpar{MR2219712} via a T-estimator based on the second sample $\bm{X}_2$. Then Theorem~9 of that paper applies with the parameters $\Sigma=2/(e-1)$, $\lambda=1$, $q=2$, $d=h$, $\kappa=4$ and $a=p/4$. It follows that the selection procedure results in an estimator $\widehat{s}$ which satisfies
\[
C\E\!\cro{h^{2}(s,\widehat s)\,\left|\,\bm{X}_1\right.}\le\min\left\{\inf_{j\ge1}\cro{h^{2}\!\left(s,\widehat{s}_{j}(\bm{X}_1)\right)+(j/p)},\inf_{k\ge1}\cro{h^{2}\!\left(s,\widehat{s}_{k}(\bm{X}_1)\right)+(k/p)}\right\}.
\]
We may then take the expectation with respect to $\bm{X}_1$ and get
\[
C\E\!\cro{h^{2}(s,\widehat s)}\le\min\left\{\inf_{j\ge1}\cro{\E\!\cro{h^{2}\!\left(s,\widehat{s}_{j}(\bm{X}_1)\right)}+(j/p)},\inf_{k\ge1}\cro{\E\!\cro{h^{2}\!\left(s,\widehat{s}_{k}(\bm{X}_1)\right)}+(k/p)}\right\}.
\]
Now applying Theorems~\ref{thm-kmon} and \ref{thm-kcc} in order to bound $\E\!\cro{h^{2}\!\left(s,\widehat{s}_{j}(\bm{X}_1)\right)}$ and $\E\!\cro{h^{2}\!\left(s,\widehat{s}_{k}(\bm{X}_1)\right)}$ respectively
we derive that the two following bounds hold simultaneously: 
\[
C\E\!\cro{h^{2}(s,\widehat s)}\le\inf_{j\ge1,\,\overline{M}>0}\left[h^2\!\left(s,\FF_{j+2}(\overline{M})\right)+\left(\left[\left(\overline{M}n^{-2}\right)^{1/3}(\log n)^{2}\right]\bigvee\left[jn^{-1}(\log n)^{3}\right]\right)\right]
\]
and
\[
C\E\!\cro{h^{2}(s,\widehat s)}\le\inf_{k\ge1,\,\overline{M}>0}\left[h^2\!\left(s,\FF_{k+2}^1(\overline{M})\right)+\left(\left(\overline{M}n^{-4}\right)^{1/5}(\log n)^{12/5}\right)\bigvee\left(kn^{-1}(\log n)^3\right)\right].
\]
This is only a simple example and the same procedure could be applied to a larger family of models and preliminary estimators but we shall not insisit on that here, the important point being that we may easily extend the results we got for a single model to large families of models and get a final bound corresponding to the best bound among all models involved in the procedure.

An alternative selection procedure leading to the same result is described in Baraud~\citeyearpar[Section~6.2]{MR2834722}. It is also possible to avoid the splitting device by using all models 
simultaneously and a penalized $\rho$-estimator as indicated in Section~7 of Baraud~{\it et al.}~\citeyearpar{Baraud:2014qv}. Again, we would get in the end the same type of risk bounds. For simplicity, we shall not insist on this other approach here.
\section{Proofs}\label{sect-proof}
\subsection{Preliminaries}\label{prelim}
In the sequel, we shall use the following elementary properties. 
%
\begin{lem}\label{L-trans}\mbox{}

1) If $\CC$ is a VC-class of subsets of $\X$ with dimension not larger than $d$ and $A\subset\X$,
then the same holds for the class $\CC\cap A$ defined by (\ref{Eq-CA}).

2) Let $\GG$ be a class of real-valued functions on a set $\X$, $\overline g$ an extremal point of $\GG$ with degree $d(\overline g)$ and $\phi(x)=x^\alpha$ for some positive $\alpha$. Let $\FF$ be a class of non-negative functions on $\X$ such that

i) $\phi(\FF)=\{\phi(f)\,|\,f\in\FF\}\subset\GG$;

ii) there exists $\overline{f}\in\FF$ such that $\phi(\overline{f})=\overline{g}$.\\
Then $\overline f$ is extremal in $\FF$ with degree not larger than $d(\overline g)$.

3) Let $\CC$ be a class of subsets of $\X$ and $A_{1},\ldots,A_{k}$ be a partition of $\X$. If for all $j\in\{1,\ldots,k\}$, $\CC\cap A_{j}$ is a VC-class with dimension not larger than $d_{j}$ then $\CC$ is a VC-class with dimension not larger than $d=\sum_{j=1}^{k}d_{j}$.
\end{lem}
\begin{proof}
Let $B\subset\X$ be a set with cardinality $d+1$. Either $B\subset A$ and $C\cap A\cap B=C\cap B$ 
for all $C\in\CC$ so that $B$ cannot be shattered by $\CC$ or $B\cap A^c$ is not empty and cannot be of the form $C\cap A\cap B$, which proves our first statement. The second statement follows from the fact that $\CC(\FF,\overline f)=\CC(\phi(\FF),\overline g)\subset\CC(\GG,\overline g)$.
For the third one, we argue as follows: if $\CC$ could shatter $d+1$ points, there would exist some $j\in\{1,\ldots,k\}$ and $d_{j}+1$ points of $A_{j}$ that could be shattered by $\CC$ and hence by $\CC\cap A_{j}$. This would be contradictory with the fact that $\CC\cap A_{j}$ is a VC-class with dimension not larger than $d_{j}$.
\end{proof}

\subsection{Proof of Theorem~\ref{thm-main}}
For $d\in\N^*$, let $\overline \Lambda(d)=\{\overline s\in \overline \Lambda,\ d(\overline{s})=d\}$.  Since $\overline S$ is assumed to be separable, $\overline \Lambda(d)\subset\overline S$ is also separable and we may therefore choose a countable and dense subset $\Lambda(d)$ of $\overline \Lambda(d)$ for each  $d\in\N^*$. Let us now choose a countable and dense subset $S$ for $\overline S$. Possibly changing $S$ into $\left(\bigcup_{d\ge1}\Lambda(d)\right)\bigcup S$, we may assume with no loss of generality that $\Lambda(d)\subset S$ for all $d\in\N^*$. Finally, we define our estimator as (any) $\rho$-estimator $\widehat s$ of $s$ based on $S$ following the construction described in Section~4.2 of Baraud {\it et al.}~\citeyearpar{Baraud:2014qv} as well as the notations of this paper.

For $y\ge 1$ and $\overline s\in\overline \Lambda$, we set 
\[
\B^{S}(s,\overline s,y)=\ac{t\in S\,\left|\,h^{2}(s,t)+h^{2}(s,\overline s)\le 
y^{2}/n\right.}.
\]
Note that $\B^{S}(s,\overline s,y)$ may be empty. We start our proof with the following lemma.
\begin{lem}\label{lem1}
For all $y\ge 1$ and $\overline s\in\overline \Lambda$
\[
\FF(S,\overline s,y)=\ac{\psi\left(\sqrt{{t}/\overline{s}}\right),\ {t}\in \B^{S}(s,\overline s,y)}
\]
is a weak VC-major class with dimension not larger than $d(\overline s)$.
\end{lem}

\begin{proof}
Since $(S/\overline s)$ is weak VC-major with dimension not larger than $d(\overline s)$ and the map $x\mapsto \psi(\sqrt{x})$ is increasing from $[0,+\infty]$ to $[-1,1]$, it follows from Baraud~\citeyearpar[Proposition~3]{Baraud:2014kq} that
\[
\FF'(S,\overline s,y)=\ac{\psi\left(\sqrt{{t}/\overline{s}}\right),\ t\in S}
\]
is weak VC-major with dimension not larger than $d(\overline s)$ and so is $\FF(S,\overline s,y)\subset \FF'(S,\overline s,y)$.
\end{proof}
Let us now go on with the proof of Theorem~\ref{thm-main}. We fix $y\ge 1$, $\overline s\in\overline{\Lambda}$ and $d=d(\overline{s})$. It follows from Baraud~\citeyearpar[Proposition 3 on page 386 with $\psi/\sqrt{2}$ in place of $\psi$]{MR2834722} and the definition of $ \B^{S}(s,\overline s,y)$ that, for all $t\in \B^{S}(s,\overline s,y)$, 
\[
\E\!\cro{\psi^{2}\pa{\sqrt{{t\over \overline s}(X_{1})}}}\le \cro{6\left(h^{2}(s,t)+h^{2}(s,\overline s)\right)}\wedge 1\le \pa{6y^{2}\over n}\wedge 1.
\]
Since $S$ is countable and $\psi$ bounded by 1, the family $\FF(S,\overline s,y)$ is also countable  and its elements are bounded by 1. Besides, Lemma~\ref{lem1} ensures that $\FF(S,\overline s,y)$ is a weak VC-major class with dimension not larger than $d\ge 1$. We may therefore apply Corollary~1 of Baraud~\citeyearpar{Baraud:2014kq} to the family $\FF(S,\overline s,y)$ with $b=1$, $\sigma^{2}=(6y^{2}/n)\wedge 1$ and get
\begin{eqnarray*}
\w^{S}(\gs,\overline \gs,y)&=&\E\!\cro{\sup_{f\in \FF(S,\overline s,y)}\ab{\sum_{i=1}^{n}\pa{f(X_{i})-\E\!\cro{f(X_{i})}\strut}}}\\
&\le& \left[4\sqrt{2n\overline \Gamma(d)}\times \sigma\log\pa{e\over\sigma}\right]+16\overline \Gamma(d)\\
&\le& \left[8\sqrt{3\overline \Gamma(d)}\times y\log\pa{e\pa{{n\over 6y^{2}}\vee 1}}\right]+16\overline \Gamma(d),
\end{eqnarray*}
with
\[
\overline \Gamma(d)=\log\pa{2\sum_{j=0}^{d\wedge n}\binom{n}{j}}\le\widetilde{\Gamma}(d)=
\log 2+(d\wedge n)\log\pa{en\over d\wedge n}.
\]
In particular, if $y^{2}\ge\widetilde{\Gamma}(d)/6\ge (d\wedge n)/6$ then $\widetilde{\Gamma}(d)\le y\sqrt{6\widetilde{\Gamma}(d)}$, hence
\begin{equation}
\w^{S}(\gs,\overline \gs,y)\le 8y\sqrt{3\widetilde{\Gamma}(d)}\cro{\log\pa{{en\over n\wedge d}}+2\sqrt{2}}
\quad\mbox{for }y\ge\sqrt{\widetilde{\Gamma}(d)/6}.
\label{Eq-ws}
\end{equation}
We recall that the quantity $D^{S}(\gs,\overline \gs)$ is defined in Section~4.3 of Baraud~{\it et al.}~\citeyearpar{Baraud:2014qv} by 
\[
D^{S}(\gs,\overline \gs)=y_{0}^{2}\vee1\quad\mbox{with}\quad y_{0}=\sup\ac{y\ge0\,\left|\,\w^{S}(\gs,\overline \gs,y)>c_0y^{2}\right.}\quad\mbox{and}\quad c_0=\frac{\sqrt{2}-1}{2\sqrt{2}}.
\]
It follows from (\ref{Eq-ws}) that $c_{0}y^{2}<\w^{S}(\gs,\overline \gs,y)$ implies that either $y<\sqrt{\widetilde{\Gamma}(d)/6}$ or 
\[
y\le (c_{0}y)^{-1}\w^{S}(\gs,\overline \gs,y)\le 8c_{0}^{-1}\sqrt{3\widetilde{\Gamma}(d)}\cro{\log\pa{{en\over n\wedge d}}+2\sqrt{2}}=B.
\]
Since in both cases, $y^{2}\le \max\left\{{\widetilde{\Gamma}(d)/6};B^{2}\right\}=B^{2}$ and $d=d(\overline{s})$, we deduce that
\begin{equation}\label{eq-bDS}
D^{S}(\gs,\overline \gs)\le B^2\le\overline{ \kappa}[d(\overline{s})\wedge n]\log^{3}\left({en\over d(\overline{s})\wedge n}\right)\le \overline{\kappa}d(\overline{s})\log_+^{3}\left(\frac{n}{d(\overline{s})}\right)\end{equation}
for all $s\in\LL_{\mu}$, $\overline{s}\in S$ and some positive numerical constant $\overline{\kappa}$. We now use Theorem~1 in Baraud~{\it et al.}~\citeyearpar{Baraud:2014qv} for which we recall that the notation $\gh^{2}(\gt,\gt')$ defined for densities $t,t'\in\LL_\mu$ means $nh^{2}(t,t')$. Since $\bigcup_{d\ge 1}\Lambda(d)\subset S$, we obtain that for all $\xi>0$, with probability at least $1-e^{-\xi}$,
\begin{equation}
C'h^{2}(s,\widehat s)\le \inf_{\overline s\in S}\cro{h^{2}(s,\overline s)+{D^{S}(\gs,\overline \gs)\over n}}+{\xi \over n}\le\inf_{d\ge1}\cro{\inf_{\overline s\in\Lambda (d)}h^{2}(s,\overline s)+\overline{\kappa} {d\over n}\log_+^{3}\pa{n\over d}}+{\xi \over n}.
\label{et-3}
\end{equation}
Finally, $\Lambda(d)$ being dense in $\overline \Lambda(d)$, 
\[
\inf_{\overline s\in\Lambda (d)}h^{2}(s,\overline s)=\inf_{\overline s\in\overline \Lambda (d)}h^{2}(s,\overline s)\quad\mbox{for\ all }d\in\N^*
\]
and the bracketed term on the right-hand side of~\eref{et-3} becomes
\[
\inf_{d\ge1}\cro{\inf_{\overline s\in\overline \Lambda (d)}h^{2}(s,\overline s)+\overline{\kappa}{d\over n}\log_+^{3}\pa{n\over d}}=
\inf_{\overline s\in \overline \Lambda}\cro{h^{2}(s,\overline s)+\overline{\kappa} {d(\overline s)\over n}\log_+^{3}\pa{n\over d(\overline s)}}.
\]
Our conclusion follows.

\subsection{Proof of Theorem~\ref{thm-mono}}
Let $D\ge 1$, $\lambda\ge 0$ and $\overline s\in \overline V(D)$ be based on the partition $\II\in\JJ(D+2)$. For all $t=f\1_{(\overline x,+\infty)}$ and $I\in\II$, the positive part $(t-\lambda\overline s)_{+}$ of  $t-\lambda\overline s$ is 0 on $(-\infty,\overline x]\cap I$ and is non-increasing on $I\cap (\overline x,+\infty)$. Consequently, $\{t-\lambda\overline s>0\}\cap I=\{(t-\lambda\overline s)_{+}>0\}\cap I\cap (\overline x,+\infty)$ is a sub-interval of $I$ (possibly empty) and $\CC(\overline S,\overline s,\lambda)\cap I$ is therefore VC with dimension not larger than 2. By Lemma~\ref{L-trans}, $\CC(\overline S,\overline s,\lambda)$ is VC with dimension not larger than $2(D+2)$ and by Proposition~\ref{prop-fonda} the element $\overline s$ is extremal in $\overline S$ with dimension not larger than $4(D+2)$. Finally Theorem~\ref{thm-mono} follows from Theorem~\ref{thm-main}. 

\subsection{Proof of Proposition~\ref{prop-approx}}
It relies on a series of approximation lemmas that shall also prove useful in the sequel.
\begin{lem}\label{lem-meanapprox}
Let $f$ be a monotone function with finite variation $V_I(f)$ on some interval $I$ of finite length $l$. Then 
\[
\int_I\left[f(x)-\overline{f}\right]^2dx\le \frac{l\left[V_I(f)\right]^2}{4}\qquad\mbox{with}\qquad
\overline{f}=\frac{1}{l}\int_If(x)\,dx
\]
and the factor $1/4$ is optimal.
\end{lem}
\begin{proof}
Assuming, without loss of generality, that $f$ is non-increasing, let us observe that one can replace 
$f$ by $g$ with $g(x)=f(x-c)-\overline{f}$ where $c$ is the left-hand point of $I$. This
amounts to assume that $c=0=\overline{f}$. Let $f(0+)=a$, $f(l-)=-b$, $a+b=V_I(f)=R$ and $\lambda=l^{-1}\sup\{x\,|\,f(x)>0\}\in(0,1)$. Then
\[
\int_0^{\lambda l} f(x)\,dx=-\int_{\lambda l}^l f(x)\,dx=Al\le l\min\{a\lambda,b(1-\lambda)\}
=l\min\{(R-b)\lambda,b(1-\lambda)\}.
\]
A maximization with respect to $b$ and $\lambda$ shows that $A\le R/4$ and it follows that
\[
\int_0^lf^2(x)\,dx=\int_0^{\lambda l}f^2(x)\,dx+\int_{\lambda l}^lf^2(x)\,dx\le(a+b)Al=RAl\le \frac{lR^2}{4}.
\]
The optimality follows by considering the case of $f=(R/2)\left(\1_{(0,l/2]}-\1_{(l/2,l)}\right)$.
\end{proof}
Our next lemma involves the norm in $\IL_{2}(\R,\B(\R),dx)$ hereafter denoted by $\norm{\cdot}$. 
\begin{lem}\label{lem-debase}
Let $f$ be a non-increasing function on $(a,b)$ with finite variation $V_{(a,b)}(f)<R$. For all $D\ge 1$, there exists a partition $\II$ of $(a,b)$ into at most $D$ intervals and a function $f_{\II}$ which is piecewise constant on each element of the partition $\II$ and non-increasing such that $f(b-)\le f_{\II}\le f(a+)$, $\norm{f_{\II}\1_{(a,b)}}\le\norm{f\1_{(a,b)}}$ and 
\begin{equation}\label{eq-approx1}
\int_a^bf_{\II}(x)\,dx=\int_a^bf(x)\,dx,\qquad\norm{(f-f_{\II})\1_{(a,b)}}\le\frac{R\sqrt{b-a}}{2D}.
\end{equation}
Besides, there exists a partition $\II'$ of $(a,b)$ into at most $2D$ intervals of length not larger than $(b-a)/D$ such that for all $I\in\II'$, $V_{I}(f)\le RD^{-1}$. The same results hold for non-decreasing functions on $(a,b)$.
\end{lem}
\begin{proof}
Clearly, the results remain valid if we replace $f$ by $g$ with $g=f$ almost everywhere (with respect to the Lebesgue measure), $g(a+)=f(a+)$ and $g(b-)=f(b-)$. Since $f$ is non-increasing on $(a,b)$, for all $x\in(a,b)$ $f(x+)$ exists and $f$ admits an at most countable number of discontinuities. We may therefore assume that $f$ is actually defined on $[a,b]$, right-continuous on $[a,b)$ and left-continuous at $b$.  

Starting from $x_{0}=a$, define recursively for all $j\ge 1$, 
\[
x_{j}=\sup\ac{x\in [x_{j-1},b],\ f(x_{j-1})-f(x)\le RD^{-1}}.
\] 
If $k\ge 1$ and $x_{k}<b$, $f(x_{k-1})-f(x)>RD^{-1}$ for all $x>x_{k}$ hence  $f(x_{k})-f(x_{k-1})\ge RD^{-1}$ since $f$ is right-continuous. In particular for such a $k$, we necessarily have 
\[
R>f(a)-f(x_{k})=\sum_{j=1}^{k}f(x_{j-1})-f(x_{j})\ge kRD^{-1},
\]
which implies that $k<D$. The process therefore results in a finite number of distinct points $x_{0}=a<x_{1}<\ldots<x_{K+1}=b$ with $K+1\le D$. It also follows from the definition of the $x_j$ that $f(x_{j-1})-f(x_{j}-)\le RD^{-1}$ for $1\le j\le K+1$.
Let us now set 
\[
\overline{f}_j=(x_j-x_{j-1})^{-1}\int_{x_{j-1}}^{x_j}f(x)\,dx\qquad\mbox{and}\qquad
f_{\II}=\sum_{j=1}^{K+1}\overline{f}_j\1_{(x_{j-1},x_{j}]}.
\]
Note that $f(b-)\le f_{\II}\le f(a+)$, $\int_a^bf_{\II}(x)\,dx=\int_a^bf(x)\,dx$ and that $f_{\II}$ is non-increasing and piecewise constant on a partition of $(a,b)$ into $K+1$ intervals. Since, for all $j$, $0\le f(x_{j-1})-f(x_j-)\le RD^{-1}$, it follows from Lemma~\ref{lem-meanapprox} that
\[
\norm{(f-f_{\II})\1_{(a,b)}}^{2}=\sum_{j=1}^{K+1}\int_{x_{j-1}}^{x_{j}}\pa{f-\overline{f}_j}^{2}dx\le
\left(\frac{R}{2D}\right)^2\,\sum_{j=1}^{K+1}\left(x_j-x_{j-1}\right)=\left(\frac{R}{2D}\right)^2(b-a).
\]
Moreover Jensen's Inequality implies that
\[
\int_{x_{j-1}}^{x_{j}}f_{\II}^2(x)\,dx=\left(x_j-x_{j-1}\right)\left(\frac{1}{x_j-x_{j-1}}\int_{x_{j-1}}^{x_{j}}f(x)\,dx\right)^2\le\int_{x_{j-1}}^{x_{j}}f^2(x)\,dx,
\]
which shows that $\norm{f_{\II}\1_{(a,b)}}\le\norm{f\1_{(a,b)}}$ and proves the first part of the lemma. 

For the second part, define $\II'$ as follows: for each element $I\in\II$ with length $\ell(I)$ larger than $(b-a)/D$ divide $I$ into $\lceil D \ell(I)/(b-a)\rceil$ intervals of length not larger than $(b-a)/D$. The process results in a new partition $\II'$ thinner than $\II$ and its cardinality is not larger than
\[
\sum_{I\in\II}\left\lceil \frac{D \ell(I)}{b-a}\right\rceil\le \sum_{I\in\II} \left[\frac{D \ell(I)}{b-a}+1\right]\le\left[\frac{D}{b-a} \sum_{I\in\II}\ell(I)\right]+|\II|\le 2D.
\]
Since by construction $V_{I}(f)\le RD^{-1}$ for all $I\in\II$, this property is also true for the elements $I$ of the partition $\II'$ which is thinner than $\II$. For non-decreasing functions, change $f$ to $-f$.
\end{proof}
%
\begin{lem}\label{L-Hell}
Given two probability densities $t,u$ with respect to $\mu$,
\begin{equation}
h(t,u)\le\left\|\sqrt{t}-\lambda\sqrt{u}\right\|\quad\mbox{for all }\lambda\in\R.
\label{Eq-proj}
\end{equation}
In particular, if $f$ is a non-negative element in $\IL_2(\mu)$ such that $\|f\|>0$ and $u=(f/\|f\|)^2$,  $h(t,u)\le\left\|\sqrt{t}-f\right\|$ for any probability density $t$ with respect to $\mu$.
\end{lem}
\begin{proof}
We notice that $\sqrt{t}$ and $\sqrt{u}$ are two vectors of norm one in $\IL_2(\mu)$ and their scalar product is $\int_{\X}\sqrt{ut}\,d\mu=\cos\alpha$ for some $\alpha\in [0,\pi/2]$. It implies that $v=\cos \alpha \sqrt{u}$ is the orthogonal projection of $\sqrt{t}$ on the linear space generated by $\sqrt{u}$, hence
\[
\inf_{\lambda\in\R}\norm{\sqrt{t}-\lambda\sqrt{u}}=\norm{\sqrt{t}-\cos \alpha \sqrt{u}}=\sin \alpha.
\]
Inequality~\eref{Eq-proj} follows from the fact that $h^{2}(t,u)=1-\cos\alpha\le \sin^{2}\alpha$ for all $\alpha\in [0,\pi/2]$. The last result is obtained from (\ref{Eq-proj}) with $\lambda=\|f\|$.
\end{proof}
To complete the proof of Proposition~\ref{prop-approx}, we apply Lemma~\ref{lem-debase} with $f=\sqrt{t}$ and $R>\sqrt{t(a+)}-\sqrt{t(b-)}$. The resulting function $f_{\II}$ is then nonnegative, non-increasing on $(a,b)$ and satisfies $0<\norm{f_{\II}}$. Setting $\overline s_{\II}=f_{\II}^{2}/\norm{f_{\II}}^{2}$, which is an element of $\overline V(D)$,  we may apply the last part of Lemma~\ref{L-Hell} with $f=f_{\II}$ which gives $h(t,\overline s_{\II})\le\norm{f-f_{\II}}\le R\sqrt{b-a}/(2D)$. The conclusion follows by letting $R$ converge to $V_{[a,b]}\left(\sqrt{t}\right)$.
\subsection{Proof of Proposition~\ref{P-approkM}}
Let $t$ be based on $\II=\{I_0,\ldots,I_{k+1}\}$, $R_j>V_{I_j}\left(\sqrt{t}\right)$ for $1\le j\le k$ and let $D_1,\ldots,D_k$ be positive integers. On the intervals $I_0$ and $I_{k+1}$, $t$ is equal to 0 and, for all other intervals of $\II$, one can apply Lemma~\ref{lem-debase}
to find an approximation $f_j$ of $\sqrt{t_j}$ which is monotone, piecewise constant with $D_j$ pieces on $I_j$ and satisfies, according to (\ref{eq-approx1}), $\left\|\left(f_j-\sqrt{t}\right)\1_{I_j}\right\|
\le R_j\sqrt{\ell(I_j)}/(2D_j)$. Therefore, if $f=\sum_{j=1}^kf_j\1_{I_{j}}$ and $u=(f/\|f\|)^2$, we derive from
Lemma~\ref{L-Hell} that
\[
h^2(t,u)\le\left\|f-\sqrt{t}\right\|^2=\sum_{j=1}^k\left\|\left(f_j-\sqrt{t}\right)\1_{I_j}\right\|^2\le
\sum_{j=1}^k\frac{\ell(I_j)R_j^2}{4D_j^2}=M.
\]
Moreover, we can always assume (modifying it on a negligeable set if necessary) that
$u$ belongs to $\overline{V}(D')\cap\FF_{k+2}$ with $D'=\sum_{j=1}^kD_j$. Given $D$, a formal
minimization with respect to the $x_j>0$ of $\sum_{j=1}^k\ell(I_j)R_j^2x_j^{-2}$
under the condition that $\sum_{j=1}^kx_j\le D$ leads to
\[
x_j=\lambda\left(\ell(I_j)R_j^2\right)^{1/3}\quad\mbox{with}\quad
\sum_{j=1}^{k}x_j=D,
\]
so that $\lambda^{-1}=D^{-1}\sum_{j=1}^k\left(\ell(I_j)R_j^2\right)^{1/3}$.
Taking into account the fact that the $D_j$ should belong to $\N^*$, we finally set
\[
D_j=\left\lceil D\left[\sum_{j=1}^k\left(\ell(I_j)R_j^2\right)^{1/3}\right]^{-1}\left(\ell(I_j)R_j^2\right)^{1/3}\right\rceil,
\]
which implies that $\sum_{j=1}^{k}D_j\le D+k$ and
\[
M\le\frac{1}{4D^2}\left[\sum_{j=1}^k\left(\ell(I_j)R_j^2\right)^{1/3}\right]^{2}
\left[\sum_{j=1}^k\left(\ell(I_j)R_j^2\right)^{1/3}\right]=\frac{1}{4D^2}
\left[\sum_{j=1}^k\left(\ell(I_j)R_j^2\right)^{1/3}\right]^3.
\]
The corresponding function $u$ belongs to $\overline{V}(D+k)\cap\FF_{k+2}$ so that
\[
h^2\!\left(t,\overline{V}(D+k)\cap \FF_{k+2}\right)\le\frac{1}{4D^2}\left[\sum_{j=1}^k\left(\ell(I_j)R_j^2\right)^{1/3}\right]^3.
\]
The conclusion follows by letting each $R_j$ converge to $V_{I_j}\left(\sqrt{t}\right)$.
\subsection{Proof of Proposition~\ref{P-approcc}}
It relies on the following approximation lemma.
%
\begin{lem}\label{L-conc}
Let $f$ be a continuous and either convex or concave function on $[a,b]$ with right-hand derivative $f'$ on $(a,b)$ satisfying $V_{(a,b)}(f')<+\infty$. The affine function $g$ on $[a,b]$ defined by $g(a)=f(a)$ and $g(b)=f(b)$ satisfies
\[
\sup_{a\le x\le b}|f(x)-g(x)|\le\frac{b-a}{4}V_{(a,b)}(f').
\]
The factor $1/4$ is optimal.
\end{lem}
\begin{proof}
Changing $f$ into $-f$, we may assume that $f$ is concave on $[a,b]$. In particular, $h(x)=f(x)-g(x)\ge0$ for $x\in [a,b]$ and since $h$ is continuous on $[a,b]$ and satisfies $h(a)=h(b)=0$, there exists some $c\in (a,b)$ such that 
\[
\sup_{a\le x\le b}h(x)=h(c)=\int_{a}^{c}\left(f'(u)-\ell\right)du=\int_{c}^{b}\left(\ell-f'(u)\right)du
\quad\mbox{with}\quad\ell={f(b)-f(a)\over b-a}.
\]
The function $f'$ being non-increasing on $(a,b)$,
\[
h(c)\le \left[(f'(a+)-\ell)(c-a)\right]\wedge \left[(\ell-f'(b-))(b-c)\right]
\]
and consequently, 
\begin{eqnarray*}
h(c)&\le& \left[\left(f'(a+)-\ell\right)(c-a){b-c\over b-a}\right]+\left[\left(\ell-f'(b-)\right)(b-c){c-a\over b-a}\right]\\
&=&  {(c-a)(b-c)\over b-a}\left[f'(a+)-\ell+\ell-f'(b-)\right]\\
&=& (b-a)\left[{c-a\over b-a}\left(1-{c-a\over b-a}\right)\right]\left[f'(a+)-f'(b-)\right]\;\;\le\;\;
 \frac{b-a}{4}V_{(a,b)}(f').
\end{eqnarray*}
The constant $1/4$ cannot be improved since it is reached for $f(x)=1-|x|$ on $[-1,1]$.
\end{proof}

Let $f'$ be the function of Lemma~\ref{L-conc} and $R>V_{(a,b)}(f')$. By Lemma~\ref{lem-debase}, one can partition $(a,b)$ into $K\le2D$ intervals $J_j$, $1\le j\le K$ of length not larger than $D^{-1}(b-a)$ with $V_{J_j}(f')< RD^{-1}$. Using this partition to approximate $f$ by a piecewise affine function $g_K$ with $K$ pieces and applying Lemma~\ref{L-conc}, we derive that
\[
\sup_{a\le x\le b}|f(x)-g_K(x)|\le(1/4)RD^{-1}[(b-a)/D]=(R/4)(b-a)D^{-2},
\]
hence
\[
\int_a^b|f(x)-g_K(x)|^2\,dx\le(R/4)^2(b-a)^3D^{-4}.
\]
Note that, by construction, $g_K$ is concave on $[a,b]$ if $f$ is and $g_K$ is convex in the opposite case.
Since $\sqrt{t}$ satisfies the assumptions of Lemma~\ref{L-conc} on each
of the $k$ non-extremal intervals of the partition $\II$ that defines $t$ and is zero on the two extremal intervals, we may use the previous approximation method on each non-extremal interval with $f=\sqrt{t}$ to get an approximation $v$ of $\sqrt{t}$ with $D'=2\sum_{j=1}^kD_j$ pieces and such that
\[
\left\|\sqrt{t}-v\right\|^2\le\frac{1}{16}\sum_{j=1}^kR_j^2[\ell(I_j)]^3D_j^{-4}\quad\mbox{if}\quad
R_j>V_{I_j}\left(\left(\sqrt{t}\right)'\right)\mbox{ for }1\le j\le k.
\]
Renormalizing $v$ as in Lemma~\ref{L-Hell}, we conclude that there exists $u$ which belongs to $\overline V_1(D')\cap\FF^1_{k+2}$ and
\[
h^2(t,u)\le M=\frac{1}{16}\sum_{j=1}^kR_j^2[\ell(I_j)]^3D_j^{-4}.
\]
We now mimic the proof of Proposition~\ref{P-approkM} to optimize the $D_j$ and get
\[
D_j=\left\lceil D\left[\sum_{j=1}^k\left([\ell(I_j)]^3R_j^2\right)^{1/5}\right]^{-1}\left([\ell(I_j)]^3R_j^2\right)^{1/5}\right\rceil,
\]
so that finally $\sum_{j=1}^{k}D_j\le D+k$ and
\[
M\le\frac{1}{16D^4}\left[\sum_{j=1}^k\left([\ell(I_j)]^3R_j^2\right)^{1/5}\right]^{4}
\left[\sum_{j=1}^k\left([\ell(I_j)]^3R_j^2\right)^{1/5}\right]=\frac{1}{16D^4}
\left[\sum_{j=1}^k\left([\ell(I_j)]^3R_j^2\right)^{1/5}\right]^5.
\]
The corresponding function $u$ belongs to $\overline{V}_1\!\left(2(D+k)\strut\right)\cap\FF^1_{k+2}$ so that
\[
h^2\!\left(t,\overline{V}_1\!\left(2(D+k)\cap\FF^1_{k+2}\right)\right)\le\frac{1}{16D^4}\left[\sum_{j=1}^k\left([\ell(I_j)]^3R_j^2\right)^{1/5}\right]^5.
\]
The conclusion follows by letting $R_j$ converge to $V_{I_j}\left(\left(\sqrt{t}\right)'\right)$ for each $j$.

\end{document}